\documentclass[twoside,a4paper,12pt,centertags]{amsart}
\usepackage{amsmath,amssymb,verbatim,vmargin}
\usepackage[bookmarks=true,pagebackref=true]{hyperref}   

\allowdisplaybreaks 
\DeclareMathOperator*{\esssup}{ess\,sup}

\theoremstyle{plain}
\newtheorem{thm}{Theorem}[section]
\newtheorem{lem}[thm]{Lemma}

\theoremstyle{definition}
\newtheorem{defn}[thm]{Definition}
\newtheorem{rem}[thm]{Remark}

\newtheorem{obs}[thm]{Observation}

\title[Weighted tent spaces with Whitney averages]
{Weighted tent spaces with Whitney averages: factorization, interpolation and duality}
\author{Yi Huang}
\address{Univ. Paris-Sud, Laboratoire de Math\'ematiques, UMR 8628 du CNRS, F-91405 Orsay}
\email{Yi.Huang@math.u-psud.fr}
\date{\today}
\subjclass[2010]{42B35, 46E30.} 
\keywords{Tent spaces, Whitney averages, strong factorization, Calder\'on's product,
quasi-Banach complex interpolation, multipliers and duality theory.}

\newcommand{\bRt}{\mathbb{R}^{n+1}_+}
\newcommand{\bRn}{\mathbb{R}^n}
\newcommand{\tps}{\texorpdfstring}
\newcommand{\loc}{\text{{\rm loc}}}
\newcommand{\pd}{\partial}

\newcommand{\bR}{{\mathbb R}}
\newcommand{\bC}{{\mathbb C}}
\newcommand{\bN}{{\mathbb N}}
\newcommand{\cA}{{\mathcal A}}
\newcommand{\cC}{{\mathcal C}}
\newcommand{\cM}{{\mathcal M}}
\newcommand{\cN}{{\mathcal N}}
\newcommand{\cV}{{\mathcal V}}
\newcommand{\cW}{{\mathcal W}}
\newcommand{\cX}{{\mathcal X}}
\newcommand{\cY}{{\mathcal Y}}
\newcommand{\EP}{\mathcal{E}\,} 

\newcommand{\wt}{\widetilde}
\newcommand{\wh}{\widehat}

\begin{document}

\begin{abstract}
In this paper, we introduce a new scale of tent spaces which covers, 
the (weighted) tent spaces of Coifman-Meyer-Stein and of Hofmann-Mayboroda-McIntosh, 
and some other tent spaces considered by 
Dahlberg, Kenig-Pipher and Auscher-Axelsson in elliptic equations.  
The strong factorizations within our tent spaces, 
with applications to quasi-Banach complex interpolation and to multiplier-duality theory, 
are established. This way, we unify and extend the corresponding results obtained by 
Coifman-Meyer-Stein, Cohn-Verbitsky and Hyt\"onen-Ros\'en.
\end{abstract}

\maketitle

\section*{0. Basic notations and article structure}
Let $\bRt=\bRn\times\bR_+=\bRn\times(0,\infty)$ be the usual upper half space in $\bR^{n+1}$.
Points in $\bRn$ (respectively in $\bRt$) will be generally denoted by the letters 
$x$ or $z$ (respectively by $(y,t)$ or $(z,s)$). For a point $(y,t)$ in $\bRt$, 
we let $B(y,t)=\{z\in\bRn\mid |z-y|<t\}$ lie in the boundary $\bRn=\pd\bRt$.
Here and below, the capital letter $B$ denotes an open ball in $\bRn$,
and $|\cdot|$ denotes the Euclidean distance on $\bRn$.

Given $\alpha>0$, 
we shall denote the cone, of aperture $\alpha$ and with vertex $x\in\bRn$, by 
$$\Gamma_{\alpha}(x):=\{(y,t)\in\bRt\mid|y-x|<\alpha t\}=
\{(y,t)\in\bRt\mid B(y,\alpha t)\ni x\},$$ 
and shall denote the tent, of aperture $\alpha$ and with base $B\subset\bRn$, by 
$$\wh{B^\alpha}:=\bigg(\bigcup_{x\in B^c}\Gamma_{\alpha}(x)\bigg)^c=
\{(y,t)\in\bRt\mid B(y,\alpha t)\subset B\}.$$ 
If $\alpha=1$, we write the two standard objects simply as $\Gamma(x)$ and $\wh{B}$. 

Surrounding a point $(y,t)\in\bRt$, 
we construct its \textit{Whitney box} as
$$W(y,t):=\{(z,s)\in\bRt\mid|z-y|<\alpha_1 t,\alpha_2^{-1}t<s<\alpha_2t\}.$$
Here, the two numbers $(\alpha_1,\alpha_2)$ with $\alpha_1>0$ and $\alpha_2>1$, 
are called the \textit{Whitney parameters}. 
They are said to be \textit{consistent} if $0<\alpha_1<\alpha_2^{-1}<1$. 

Throughout this article,
the set of Vinogradov notations $\{\lesssim, \simeq, \gtrsim\}$ will be used.
For two quantities $a$ and $b$, which can be function values, 
set volumes, function norms or anything else, the term $a\lesssim b$ means that
there exists a constant $C>0$, which depends on parameters at hand,
such that $a\leq C b$. In a similar way, $a\gtrsim b$ means $b\lesssim a$, 
and, $a\simeq b$ means both $a\lesssim b$ and $a\gtrsim b$.

This paper is organized as follows. 

\begin{itemize}
\item Section \ref{sec:defn}. We define in Definition \ref{defn:tentspaces}
our scale of tent spaces $T^{p,r}_{q,\beta}$ systematically. 
At the end of this section, we will also discuss some basic function space properties, 
such as convexity and separability, of these new tent spaces.

\item Section \ref{sec:coin}. We show that the definition of our tent spaces 
is independent of the aperture used in cones and tents, 
and of the pair of Whitney parameters used in Whitney boxes. As a reward, 
we can see for $r=q$, the coincidence (Theorem \ref{thm:coincidence}) 
of our tent spaces with the classical tent spaces of Coifman-Meyer-Stein
and the weighted tent spaces of Hofmann-Mayboroda-McIntosh.

\item Section \ref{sec:fac} and Section \ref{sec:proof}. 
The core (endpoint) factorization theorem (Theorem \ref{thm:factorizationend}) 
is presented in Section \ref{sec:fac}, 
with its full proof postponed to Section \ref{sec:proof}.
Together with a multiplication lemma,
we show conditionally on Theorem \ref{thm:factorizationend}
the general multiplication and factorization theorem (Theorem \ref{thm:factorization}).

\item Section \ref{sec:inter} and Section \ref{sec:muldual}.
Under the general multiplication and factorization theorem, 
the quasi-Banach complex interpolation
(Theorem \ref{thm:Interpolation}) and the multiplier-duality results 
(Theorem \ref{thm:multiplier} and Theorem \ref{thm:duality}) 
will be established in Section \ref{sec:inter} and Section \ref{sec:muldual} respectively. 
There, we will also make a detailed connection with the corresponding known results on 
interpolation, multiplication, factorization and duality
of tent spaces, which are mainly obtained by Coifman-Meyer-Stein, 
Cohn-Verbitsky and Hyt\"onen-Ros\'en.
\end{itemize}

\section{Definitions of the tent spaces \tps{$T^{p,r}_{q,\beta}$}{}} \label{sec:defn}
Let $r\in(0,\infty]$. 
By $L^r_{\loc}(\bRt;\bC)$, we mean the class of complex-valued measurable functions 
which are defined on $\bRt$ and locally in $L^r$. 
This interpretation also makes sense when $r=\infty$.
For $r\in(0,\infty)$ and $f\in L^r_{\loc}(\bRt;\bC)$, 
denote the (unweighted) \textit{$L^r$-Whitney average} of $f$ on $W(y,t)$ by 
$$\cW_r(f)(y,t):=|W(y,t)|^{-1/r}\|f\|_{L^r(W(y,t),\,dzds)},$$
while for $r=\infty$, we take the usual essential supremum interpretation
$$\cW_\infty(f)(y,t):=\esssup_{(z,s)\in W(y,t)}|f(z,s)|.$$
Note that we use the curled $\cW$ to distinguish it as an averaging functional.

Here and below, apart from the Euclidean distance,
$|\cdot|$ also denotes the moduli of complex values or the set volumes in $\bRn$ and $\bRt$. 

\begin{defn} \label{defn:tentspaces}
$I)$ For $0<p,q\leq\infty$, 
we first define in $L^q_{\loc}(\bRt;\bC)$ the scale $T^p_q$ of 
\textit{tent spaces} into the following four non-overlapping categories.

$A)$ $0<p,q<\infty$. In this case, we let
$$T^{p}_q:=\{g\mid \cA_{q}(g)\in L^p(\bRn)\}\,\,\,\,\text{and}\,\,\,\,
\|g\|_{T^p_q}:=\|\cA_q(g)\|_{L^p},$$
where the \textit{conical $q$-functional} $\cA_q$ is defined as
$$\cA_q(g)(x):=\bigg(\iint_{\Gamma(x)}|g(y,t)|^q\frac{dydt}{t^{n+1}}\bigg)^{1/q},\,x\in\bRn.$$

$B)$ $0<q<p=\infty$. In this case, we let 
$$T^{\infty}_q:=\{g\mid\cC_q(g)\in L^\infty(\bRn)\}\,\,\,\,\text{and}\,\,\,\,
\|g\|_{T^{\infty}_q}:=\|\cC_q(g)\|_{L^\infty},$$
where the \textit{Carleson $q$-functional} $\cC_q$ is defined as
$$\cC_q(g)(x):=\sup_{B\ni x}|B|^{-1/q}
\bigg(\iint_{\wh{B}}|g(y,t)|^q\frac{dydt}{t}\bigg)^{1/q},\,x\in\bRn.$$

$C)$ $0<p<q=\infty$. In this case, we let 
$$T^p_{\infty}:=\{g\mid\cN(g)\in L^p(\bRn)\}\,\,\,\,\text{and}\,\,\,\,
\|g\|_{T^{p}_{\infty}}:=\|\cN(g)\|_{L^p},$$
where the \textit{non-tangential maximal functional} $\cN$ is defined as
$$\cN(g)(x):=\sup_{(y,t)\in\Gamma(x)}|g(y,t)|,\,x\in\bRn.$$

$D)$ $p=q=\infty$. In this case, we simply let $T^{\infty}_{\infty}:=L^\infty(\bRt)$.\\
Let $\beta\in\bR$. We also define the scale $T^p_{q,\beta}$ of \textit{weighted tent spaces} by
$$T^{p}_{q,\beta}:=\{g\mid g(y,t)t^{-\beta}\in T^{p}_q\}\,\,\,\,\text{and}\,\,\,\,
\|g\|_{T^{p}_{q,\beta}}:=\|g(y,t)t^{-\beta}\|_{T^{p}_q}.$$

$II)$ Given $0<r\leq\infty$ and $\beta\in\bR$, 
and assume that the pair of Whitney parameters $(\alpha_1,\alpha_2)$ is consistent. 
Then corresponding to each category above,
we define in $L^r_{\loc}(\bRt;\bC)$ the scale $T^{p,r}_{q}$ 
of \textit{tent spaces with Whitney averages}
by $$T^{p,r}_{q}:=\{f\mid \cW_r(f)\in T^p_q\}\,\,\,\,\text{and}\,\,\,\,
\|f\|_{T^{p,r}_q}:=\|\cW_r(f)\|_{T^p_q},$$
and the scale $T^{p,r}_{q,\beta}$ of \textit{weighted tent spaces with Whitney averages}
by $$T^{p,r}_{q,\beta}:=\{f\mid f(z,s)s^{-\beta}\in T^{p,r}_q\}\,\,\,\,\text{and}\,\,\,\,
\|f\|_{T^{p,r}_{q,\beta}}:=\|f(z,s)s^{-\beta}\|_{T^{p,r}_q}.$$
\end{defn}

In the above definitions, the $L^r$-Whitney average and the weight $\beta$ 
are required for the applications to boundary value problems of
second order elliptic PDEs in \cite{AAInv}.
In practice the role of $\beta$ is a regularity index,
and the weight constraint $\beta\in[-2/q,0]$, 
with $\beta=0$ if $q=\infty$, is taken in \cite{AAInv}. 

\begin{rem} \label{rem:iso}
By definition $T^{p}_{q,0}=T^{p}_{q}$ and $T^{p,r}_{q,0}=T^{p,r}_{q}$.
Moreover, $T^{p}_{q,\beta}$ is isometric to $T^{p}_{q}$ and 
$T^{p,r}_{q,\beta}$ is isometric to $T^{p,r}_{q}$ via
$f\rightarrow \wt{f}$ with $\tilde{f}(z,s)=f(z,s)s^{-\beta}$.
Observe also that since $(z,s)\in W(y,t)$ implies $s\simeq t$,
we have that $f\in T^{p,r}_{q,\beta}\Longleftrightarrow \cW_r(f)\in T^{p}_{q,\beta}$.
\end{rem}

\begin{rem}  \label{rem:classicaltent}
The classical tent spaces of Coifman-Meyer-Stein in \cite{CMS}, where the weight $\beta=0$
and Category $C)$ is smaller\footnotemark
\footnotetext{More precisely, \cite{CMS} requires the additional boundary assumption 
$g\in C_{n.t.}(\bRt;\bC)$, 
meaning that $g$ is a complex-valued continuous function on $\bRt$ and 
also has \textit{non-tangential convergence}: 
$$\lim_{\Gamma(x)\ni(y,t)\rightarrow x}g(y,t)\,\,\,\text{exists for almost every}\,x\in\bRn.$$},
and the weighted tent spaces of Hofmann-Mayboroda-McIntosh in \cite{HMMc},
where only Category $A)$ is considered, 
are all included in our scale $T^p_{q,\beta}$.
The scale $T^{p,r}_{q,\beta}$ with Whitney averages covers the function spaces 
which were introduced in \cite{D} and \cite{KP}, 
and further investigated in \cite{AAInv}, \cite{HR} and \cite{M}. 
In this regard, see the concluding paragraphs of 
Section \ref{sec:muldual} for a detailed correspondence.
Note that we also bring in Category $D)$, where if $0<r<\infty$,
we call functions in $T^{\infty,r}_{\infty}$ the \textit{$r$-Whitney multipliers}.
In the trivial case $p=q=r=\infty$, it is not difficult to observe that
$T^{\infty,\infty}_{\infty}=T^{\infty}_{\infty}=L^\infty(\bRt)$.
\end{rem}

We end this section with several function space properties of our tent spaces.

\underline{Convexity and completeness}. 
Given the tent space $T^{p,r}_{q,\beta}$,
we let $\tau=\min(p,q,r)$.
Observe that when $\tau\geq1$, the space $T^{p,r}_{q,\beta}$ is Banach.
In fact, the triangle inequality simply follows from Minkowski's inequality,
and the completeness can be deduced from the one of $T^{p}_q$,
as $\cW_r(f(z,s)s^{-\beta})\in T^{p}_q$ if $f\in T^{p,r}_{q,\beta}$.
Note that we identify two measurable functions the same 
if they only differ on a set with measure 0.

\underline{Power and convexification}.
For a quasi-Banach function space, 
the trick of taking the powers is particularly useful. 
As for our tent spaces, let 
$$\big[T^{p,r}_{q,\beta}\big]^\theta:=\{f\,\text{measurable}\mid 
|f|^{1/\theta}\in T^{p,r}_{q,\beta}\},\,\theta\in (0,1),$$
equipped with $\|f\|_{[T^{p,r}_{q,\beta}]^\theta}:=
\||f|^{1/\theta}\|^\theta_{T^{p,r}_{q,\beta}}$.
This way, we have the realization 
$$\big[T^{p,r}_{q,\beta}\big]^\theta=
T^{p/\theta,r/\theta}_{q/\theta,\beta\theta},\,\theta\in (0,1).$$
Now for the quasi-Banach $T^{p,r}_{q,\beta}$, with $\tau<1$,
$\big[T^{p,r}_{q,\beta}\big]^\tau$ is then a convexification of $T^{p,r}_{q,\beta}$.

\underline{Separability and density}.
Consider the covering of $\bRt$ by rational rectangles, 
which are of the product form 
$\prod_{i=1}^{n+1}(a_i,b_i)$,
where for $1\leq i\leq n+1$, $a_i$ and $b_i$ are in $\mathbb{Q}$ and $a_{n+1}>0$.
Let $E$ be the linear span on $\mathbb{Q}$ 
of the characteristic functions of rational rectangles.
If $0<p<\infty$, one can show that the countable set
$E$ is dense in $T^{p,r}_{q,\beta}$. 
We also point out that if $\sigma=\max(p,q,r)<\infty$,
the $L^r$ functions which have compact support in $\bRt$ are dense in $T^{p,r}_{q,\beta}$.

\section{Coincidence and change of geometry} \label{sec:coin}
A demanding reader may ask two natural questions: 
i) how do the inner (local) Whitney averages $\cW_r$ 
behave under the outer (boundary-reaching) $\cA_q$ or $\cC_q$ averages?
ii) is our Definition \ref{defn:tentspaces} independent of the involved geometrical parameters?
Aiming at the question i), 
we will first investigate the relation between the classical scale $T^p_q$
and our scale $T^{p,r}_{q}$ with Whitney averages.
At the end of this section, we will also give an observation on the question ii).

Let us start with the following result.

\begin{obs}[Change of apertures] \label{obs:aperture}
Define for $0<q<\infty$ and $\alpha>0$ 
the following three $\alpha$-apertured functionals as
$$\cA^{\alpha}_{q}(g)(x):=\bigg(\iint_{\Gamma_\alpha(x)}|g(y,t)|^q
\frac{dydt}{t^{n+1}}\bigg)^{1/q},\,x\in\bRn,$$
$$\cN^{\alpha}(g)(x):=\sup_{(y,t)\in\Gamma_\alpha(x)}|g(y,t)|,\,x\in\bRn,$$
$$\cC^{\alpha}_{q}(g)(x):=\sup_{B\ni x}|B|^{-1/q}
\bigg(\iint_{\wh{B^\alpha}}|g(y,t)|^q\frac{dydt}{t}\bigg)^{1/q},\,x\in\bRn.$$
Similar to Definition \ref{defn:tentspaces}, 
these functionals can also result in a scale of tent spaces $^\alpha T^{p}_q$, 
where we let $^\alpha T^\infty_\infty=L^\infty$ for the trivial case $p=q=\infty$.
It is well known that 
we have the change of aperture equivalence $^\alpha T^{p}_q=T^{p}_q$, with
\begin{equation} \label{eqn:tentaper}
C(n,\alpha,p,q)\|g\|_{T^p_q}\leq\|g\|_{^\alpha T^{p}_q}\leq C'(n,\alpha,p,q)\|g\|_{T^p_q},
\,0<p,q\leq\infty. 
\end{equation}
For the proof, see \cite{FS} for the simple situation $0<p<q=\infty$. 
For the case $q=2$ and $0<p<\infty$ 
(hence for $0<p,q<\infty$ by taking the powers of $g$ properly),
see \cite{CMS} for a rough, and \cite{Tor} for a refined argument 
on estimating $C'$ when $\alpha>1$.
By using the atomic decomposition and the interpolation method, 
the sharp determination on both $C$ and $C'$ when $\alpha>0$, 
for the case $q=2$ and $0<p\leq\infty$,
is obtained recently in \cite{AusAngle}. 
Note that the methods of \cite{AusAngle} 
extend to the case $q=\infty$ under minor modifications.
We also remark that, the vector-valued approach in \cite{HTV} and \cite{HvNP} 
can deal with the change of apertures in a very simple manner in the Banach case, 
and then a convexification process takes care of the quasi-Banach case.
\end{obs}

\begin{thm}\label{thm:coincidence}
We have the coincidence with equivalence of quasi-norms
$$T^{p,q}_{q,\beta}= T^p_{q,\beta},\,0<p,q\leq\infty,\,\beta\in\bR.$$
In particular, $T^{p,q}_{q}= T^p_{q}$, $0<p,q\leq\infty$,
showing that the classical tent spaces 
are included in the tent spaces with Whitney averages.
\end{thm}

\begin{proof}
By Remark \ref{rem:iso}, it is enough to prove
$$T^{p,q}_{q}= T^p_{q},\,0<p,q\leq\infty.$$

We start with the following \textit{Whitney box geometry}: $\forall\,(z,s)\in\bRt$
$$W)\,\,\,\,W_*(z,s)\subset\{(y,t)| W(y,t)\ni(z,s)\}\subset W_{**}(z,s),$$
where $W_*$ and $W_{**}$ are the Whitney boxes associated to the Whitney parameters 
$(\alpha_1\alpha_2^{-1},\alpha_2)$ and $(\alpha_1\alpha_2,\alpha_2)$ 
respectively\footnotemark
\footnotetext{The pair of Whitney parameters defining $W_{**}$ is not necessarily consistent,
but for the purpose here, the consistency is not needed.}, 
and $(\alpha_1,\alpha_2)$ is the pair of Whitney parameters 
which defines $W$ and was used in Definition \ref{defn:tentspaces}.
We only need to verify the choices of $\alpha_1\alpha_2^{-1}$ and $\alpha_1\alpha_2$,
as the determination on $\alpha_2$ is straightforward.
To see the first inclusion in $W)$, given any $(y,t)\in W_*(z,s)$, we have 
$|z-y|<\alpha_1\alpha_2^{-1}s<\alpha_1 t$,
which implies $W(y,t)\ni(z,s)$. To see the second inclusion, 
given any $(y,t)$ with $W(y,t)\ni (z,s)$, we have
$|y-z|<\alpha_1t<\alpha_1\alpha_2s$, 
which implies $(y,t)\in W_{**}(z,s)$. This proves the Whitney box geometry $W)$.

For the \textit{cone geometry}, let $\alpha_{0}=\alpha_2^{-1}(1-\alpha_1)$. 
We have that: $\forall\,x\in\bRn$
$$C_1)\,\,\,\,\,(z,s)\in\Gamma_{\alpha_{0}}(x)
\,\,\text{and}\,\,(y,t)\in W_*(z,s)\Longrightarrow (y,t)\in \Gamma(x).$$
Indeed, we can compute as follow
$$|y-x|\leq |y-z|+|z-x|<\alpha_1\alpha_2^{-1}s+\alpha_2^{-1}(1-\alpha_1)s<t.$$
Let $\alpha_{C}=\alpha_2+\alpha_1\alpha_2$. There also holds: $\forall\,x\in\bRn$
$$C_2)\,\,\,\,\,(y,t)\in\Gamma(x)\,\,\text{and}\,\,
(z,s)\in W(y,t)\Longrightarrow (z,s)\in \Gamma_{\alpha_{C}}(x).$$
Indeed, we can compute as follow
$$|z-x|\leq |z-y|+|y-x|<\alpha_1t+t<(\alpha_2+\alpha_1\alpha_2)s.$$

Now from $W)+C_1)$, we have: $\forall\,x\in\bRn$
$$\chi_{\Gamma_{\alpha_{0}}(x)}(z,s)
\chi_{W_*(z,s)}(y,t)\leq \chi_{\Gamma(x)}(y,t)\chi_{W(y,t)}(z,s),$$
and from $W)+C_2)$, we have: $\forall\,x\in\bRn$
$$\chi_{\Gamma(x)}(y,t)\chi_{W(y,t)}(z,s)\leq
\chi_{\Gamma_{\alpha_{C}}(x)}(z,s)\chi_{W_{**}(z,s)}(y,t).$$
Then it follows from an integration in $(y,t)$ that: $\forall\,x\in\bRn$
$$\chi_{\Gamma_{\alpha_{0}}(x)}(z,s)\lesssim
\iint_{\bRt}\chi_{\Gamma(x)}(y,t)\frac{\chi_{W(y,t)}(z,s)}{t^{n+1}}dydt
\lesssim\chi_{\Gamma_{\alpha_{C}}(x)}(z,s),$$
where in dividing $s^{n+1}$, we use the similarity $s\simeq t$ implicitly.

If $0<q<\infty$, multiplying by $|f(z,s)|^q$ the above inequalities
and then integrating in $(z,s)$, we have from Fubini's theorem that
$$\cA^{\alpha_{0}}_q(f)(x)\lesssim\cA_q(\cW_q(f))(x)\lesssim
\cA^{\alpha_{C}}_q(f)(x),\,\forall\,x\in\bR^n.$$
If $q=\infty$, there holds a similar functional relation
$$\cN^{\alpha_{0}}(f)(x)\lesssim\cN(\cW_\infty(f))(x)\lesssim
\cN^{\alpha_{C}}(f)(x),\,\forall\,x\in\bR^n.$$
For $0<p<\infty$, taking an $L^p$ integration in $x$ in the above two functional relations 
and using the change of aperture equivalence in Observation \ref{obs:aperture}
lead us to the coincidence $T^{p,q}_{q}= T^p_q$ in Category $A)$ and Category $C)$. 

For the \textit{tent geometry}, let $\alpha_{T}=\alpha_2+\alpha_1\alpha_2^{-1}$.
We have that: $\forall\,B\subset\bRn$
$$T_1)\,\,\,\, (z,s)\in\wh{B^{\alpha_{T}}}\,\,
\text{and}\,\, (y,t)\in W_*(z,s)\Longrightarrow (y,t)\in \wh{B}.$$
Indeed, given $B\subset\bRn$, $(z,s)\in\wh{B^{\alpha_{T}}}$ and $(y,t)\in W_*(z,s)$,
then $B(z,\alpha_{T}s)\subset B$. Thus $$B(y,t)\subset B(z,t+|z-y|)
\subset B(z,t+\alpha_1\alpha_2^{-1}s)\subset B(z,\alpha_{T}s),$$ 
so $B(y,t)\subset B$.
Recall that $\alpha_{0}=\alpha_2^{-1}(1-\alpha_1)$. 
There also holds: $\forall\,B\subset\bRn$
$$T_2)\,\,\,\, (y,t)\in\wh{B}\,\,
\text{and}\,\, (z,s)\in W(y,t)\Longrightarrow (z,s)\in \wh{B^{\alpha_{0}}}.$$
Indeed, given $B\subset\bRn$, $(y,t)\in\wh{B}$ and $(z,s)\in W(y,t)$,
then $B(y,t)\subset B$. Thus $$B(z,\alpha_{0}s)\subset B(y,\alpha_{0}s+|z-y|)
\subset B(y,\alpha_{0}s+\alpha_1t)\subset B(y,t),$$ 
so $B(z,\alpha_{0}s)\subset B$ and $(z,s)\in \wh{B^{\alpha_{0}}}$.

Now from $W)+T_1)$, we have: $\forall\,B\subset\bRn$
$$\chi_{\wh{B^{\alpha_T}}}(z,s)\chi_{W_*(z,s)}(y,t)
\leq \chi_{\wh{B}}(y,t)\chi_{W(y,t)}(z,s),$$
and from $W)+T_2)$, we have: $\forall\,B\subset\bRn$
$$\chi_{\wh{B}}(y,t)\chi_{W(y,t)}(z,s)\leq 
\chi_{\wh{B^{\alpha_0}}}(z,s)\chi_{W_{**}(z,s)}(y,t).$$
Then it follows from an integration in $(y,t)$ that: $\forall\,B\subset \bRn$
$$ \chi_{\wh{B^{\alpha_T}}}(z,s)\lesssim\iint_{\bRt}
\chi_{\wh{B}}(y,t)\frac{\chi_{W(y,t)}(z,s)}{t^{n+1}}dydt
\lesssim\chi_{\wh{B^{\alpha_0}}}(z,s),$$
where in dividing $s^{n+1}$, we again use the similarity $s\simeq t$ implicitly.

If $0<q<\infty$, multiplying by $|f(z,s)|^q$ the above inequalities 
then integrating in $(z,s)$ and taking a supremum over $B\ni x$, 
we have from Fubini's theorem that
$$\cC^{\alpha_T}_q(f)(x)\lesssim\cC_q(\cW_q(f))(x)\lesssim\cC^{\alpha_0}_q(f)(x),
\,\forall\,x\in\bR^n.$$
Taking an $L^\infty$ norm in the above functional relation
and using the change of aperture equivalence in Observation \ref{obs:aperture}
lead us to the coincidence $T^{p,q}_{q}= T^p_q$ in Category $B)$. 

Together with the trivial Category $D)$, 
we can thus conclude the proof.
\end{proof}

\begin{rem}
For the coincidence with the ``classical'' tent spaces in Category $C)$, 
we mean in fact 
$T^{p,\infty}_{\infty}\cap C_{n.t.}= T^p_{\infty}\cap C_{n.t.}$, $0<p<\infty$.
\end{rem}

We end this section with another geometrical result, 
which will be needed in Section \ref{sec:proof} for the proof 
of $F_1)$ in Theorem \ref{thm:factorizationend}.

\begin{obs}[Change of Whitney parameters] \label{obs:changewhit}
Note that we have frozen two consistent parameters 
$(\alpha_1,\alpha_2)$ in Definition \ref{defn:tentspaces}.
Instead of considering different apertures as in Observation \ref{obs:aperture},
here we replace $(\alpha_1,\alpha_2)$ by another pair of consistent Whitney parameters 
$(\alpha'_1,\alpha'_2)$, with a prescribed chain condition 
$$0<\alpha_1<\alpha_1'<1/\alpha'_2<1/\alpha_2<1.$$ 
Following the way in Definition \ref{defn:tentspaces},
we can also define a scale of tent spaces associated to $(\alpha_1,\alpha_2)$.
Denoted by $^{(\alpha'_1,\alpha'_2)}T^{p,r}_{q}$, 
they should not be mistaken for the scale $^\alpha T^p_q$ in Observation \ref{obs:aperture}.
We have the change of Whitney parameters equivalence
\begin{equation} \label{eqn:changewhit}
C(\alpha_1,\alpha'_1,\alpha_2,\alpha'_2)\|f\|_{T^{p,r}_{q}}
\leq\|f\|_{^{(\alpha'_1,\alpha'_2)}T^{p,r}_{q}}
\leq C'(\alpha_1,\alpha'_1,\alpha_2,\alpha'_2)\|f\|_{T^{p,r}_{q}},
\end{equation}
where the constants $C$ and $C'$ also implicitly depend on $n$, $p$, $q$ and $r$.

The former part of this equivalence can be inspected from the chain condition satisfied by
$(\alpha_1,\alpha_2)$ and $(\alpha'_1,\alpha'_2)$. 
We prove the right hand inequality as follows.
For $(y,t)\in\bRt$, denote $\wt{W}(y,t)=B(y,\gamma_1t)\times(\gamma_2^{-1}t,\gamma_2t)$,
with $\gamma_1\geq\alpha'_1/\alpha_1$ and $\gamma_2\geq\alpha'_2/\alpha_2$.
Then one can find an integer $N=N(n,\alpha_1,\alpha_2,\alpha'_1,\alpha'_2)$ such that,
for any $(y,t)\in\bRt$, there exist $N$ points 
$\mathcal{P}_N(y,t)$ in $\wt{W}(y,t)$ with
$$\chi_{W'(y,t)}(z,s)\leq
\sum_{(\bar{y},\bar{t})\in \mathcal{P}_N(y,t)}\chi_{W(\bar{y},\bar{t})}(z,s),$$
where $W'$ is the Whitney average associated to the Whitney parameters $(\alpha'_1,\alpha'_2)$. 
Now using (\ref{eqn:tentaper}) in Observation \ref{obs:aperture}
and the geometries $\{W), C_1), C_2), T_1), T_2)\}$ in proving Theorem \ref{thm:coincidence},
there exists $\alpha=\alpha(\alpha_1,\alpha_2,\alpha'_1,\alpha'_2)$ such that
$$\|f\|_{^{(\alpha'_1,\alpha'_2)}T^{p,r}_{q}}
\lesssim \|f\|_{^\alpha T^{p,r}_{q}}\lesssim\|f\|_{T^{p,r}_{q}}.$$
We leave open the sharp determination on the bounds $C$ and $C'$ in (\ref{eqn:changewhit}).
\end{obs}

\section{Multiplication and factorization} \label{sec:fac}
The main goal of this paper, is to obtain in the spirit of \cite{CV}, 
the corresponding multiplication and factorization results 
for our new scale of tent spaces $T^{p,r}_{q,\beta}$.
Some notations and definitions in function space theory are needed. 

Denote by $\Sigma$ the $\sigma$-finite measure space $(\Omega,\mu)$,
and by $L^0$ the collection of $\mu$-measurable complex-valued functions on $\Omega$.
A \textit{quasi-Banach function lattice} $X$ on $\Sigma$ is a non-empty subspace of $L^0$, 
which is equipped with a quasi-norm $\|\cdot\|_X$ such that, 
$(X,\|\cdot\|_X)$ is complete and $X$ satisfies the \textit{lattice property}:
$$\forall\, f\in X,\,\forall\, g\in L^0, \,\text{with}\,|g|\leq |f|\,\,\,\, \mu-\text{a.e.}$$
$$\Longrightarrow\,\,\,\,  g\in X, \,\text{with}\,\|g\|_X\leq \|f\|_X.$$

Clearly, for any $f$ in a quasi-Banach function lattice $X$, $\|f\|_X=\||f|\|_X$.

\begin{defn} \label{defn:mf}
Let $\{X_i\}_{i=0}^n$ be a collection of quasi-Banach function lattices on $\Sigma$.

M) By the multiplication: $X_0\leftarrow X_1\cdots X_n$, 
we mean that for any $f_i\in X_i$, $1\leq i\leq n$, we have $f_1\cdots f_n\in X_0$ and
$$\|f_1\cdots f_n\|_{X_0}\lesssim\|f_1\|_{X_1}\cdots\|f_n\|_{X_n},$$
where the implicit constant is independent of $f_1$, $\cdots$, $f_n$.

F) By the (strong) factorization: $X_0\rightarrow X_1\cdots X_n$, 
we mean that for any $f_0\in X_0$, there exist $f_i\in X_i$, $1\leq i\leq n$, such that
$|f_0|=|f_1|\cdots |f_n|$ and $$\|f_1\|_{X_1}\cdots\|f_n\|_{X_n}\lesssim\|f_0\|_{X_0},$$
where the implicit constant does not depend on $f_0$, $f_1$, $\cdots$, $f_n$.

When M) and F) are both satisfied, we write $X_0\leftrightarrow X_1\cdots X_n$.
\end{defn}

In this paper, our central task is to prove 

\begin{thm} \label{thm:factorizationend}
For any $0<p_0,q_0,r_0\leq\infty$, we have the following factorizations
$$F_1)\,\,T^{p_0,r_0}_{q_0}\rightarrow T^{p_0,\infty}_{q_0}\cdot T^{\infty,r_0}_\infty,$$
$$F_2)\,\,T^{p_0,r_0}_{q_0}\rightarrow T^{p_0,\infty}_{\infty}\cdot T^{\infty,r_0}_{q_0},$$
$$F_3)\,\,T^{p_0,r_0}_{q_0}\rightarrow T^{p_0,\infty}_{\infty}\cdot 
T^{\infty,\infty}_{q_0}\cdot T^{\infty,r_0}_{\infty}.$$
\end{thm}

The proof of this endpoint factorization theorem will be postponed to Section \ref{sec:proof}.
Meanwhile, there holds an endpoint multiplication result.

\begin{lem} \label{lem:multiplicationend}
For any $0<p_0,q_0,r_0\leq\infty$, we have the following multiplications
$$M_1)\,\,T^{p_0}_{q_0}\leftarrow T^{p_0}_{\infty}\cdot T^{\infty}_{q_0},$$
$$M_2)\,\,T^{p_0,r_0}_{q_0}\leftarrow T^{p_0,\infty}_{\infty}\cdot 
T^{\infty,\infty}_{q_0}\cdot T^{\infty,r_0}_\infty.$$
\end{lem}

\begin{proof}[Proof of Lemma \ref{lem:multiplicationend}]
If $\max(p_0,q_0)=\infty$, there is nothing to prove for $M_1)$. If $\max(p_0,q_0)<\infty$, 
then the multiplication $M_1)$ is essentially in \cite[Lemma 2.1]{CV}.
The multiplication $M_2)$ is a consequence of H\"older's inequality and $M_1)$. 
In fact, we have
\begin{align*}
 \|fgh\|_{T^{p_0,r_0}_{q_0}}&\leq\|\cW_\infty(f)\cW_\infty(g)\cW_{r_0}(h)\|_{T^{p_0}_{q_0}}\\
&\lesssim \|\cW_\infty(f)\|_{T^{p_0}_{\infty}}
\|\cW_\infty(g)\|_{T^{\infty}_{q_0}}\|\cW_{r_0}(h)\|_{T^{\infty}_{\infty}}\\&=
\|f\|_{T^{p_0,\infty}_{\infty}}\|g\|_{T^{\infty,\infty}_{q_0}}\|h\|_{T^{\infty,r_0}_{\infty}},
\end{align*}
where $f$, $g$ and $h$ are all measurable functions on $\bRt$.

Note that for $M_1)$, the starting point of \cite[Lemma 2.1]{CV} 
is the following inequality for Carleson measures (see \cite[p. 58--61]{Stein} for example)
$$\iint_{\bRt}|f(y,t)|^p|d\mu|(y,t)\lesssim 
\|f\|_{T^p_\infty}^p\sup_{B\subset\bRn}\frac{|\mu|(\wh{B})}{|B|},$$
which holds true for any Borel measure $d\mu$ on $\bRt$ and 
any measurable $f$ such that $\cN(f)\in L^p$, $0<p<\infty$.
This is also why we define the Category $C)$ tent spaces $T^{p}_\infty$ 
without restricting them in the class $C_{n.t.}(\bRt;\bC)$.
\end{proof}

For $0<p_1,p_2\leq\infty$, define the \textit{H\"olderian triplet} 
$(p_1,p_2,(p_1,p_2)_H)$ by the relation $(p_1,p_2)_H^{-1}=p_1^{-1}+p_2^{-1}$, 
where as usual, we will admit $1/\infty=0$.

Combining $F_3)$ in Theorem \ref{thm:factorizationend} 
and $M_2)$ in Lemma \ref{lem:multiplicationend}, 
we can deduce an intermediate claim where the H\"olderian triplets enter. 

\begin{thm}\label{thm:factorization}
Suppose for $i\in\{0, 1, 2\}$, $T^{p_i,r_i}_{q_i,\beta_i}$ lies in the scale 
of weighted tent spaces with Whitney averages in Definition \ref{defn:tentspaces}.
Assume the H\"olderian relation $(H)$: 
$$p_0=(p_1,p_2)_H, \,\,q_0=(q_1,q_2)_H, \,\,r_0=(r_1,r_2)_H
\,\,\text{and}\,\, \beta_0=\beta_1+\beta_2.$$
Then we have the multiplication and factorization
$$T^{p_0,r_0}_{q_0,\beta_0}\leftrightarrow T^{p_1,r_1}_{q_1,\beta_1}
\cdot T^{p_2,r_2}_{q_2,\beta_2}.$$
\end{thm}

\begin{proof}[Proof of Theorem \ref{thm:factorization}]
By Remark \ref{rem:iso} and Definition \ref{defn:mf}, 
it is enough to assume $\beta_i=0$, $i\in\{0, 1, 2\}$. 
Thus, we are only meant to show
$$T^{p_0,r_0}_{q_0}\leftrightarrow T^{p_1,r_1}_{q_1}
\cdot T^{p_2,r_2}_{q_2}.$$

Call \textit{extremal tent spaces} those $T^{p,r}_q$ 
with at least two among $p,q,r$ equal to $\infty$.
Therefore, $T^{p_0,r_0}_{q_0}\leftrightarrow T^{p_1,r_1}_{q_1}\cdot T^{p_2,r_2}_{q_2}$
holds trivially if $T^{p_0,r_0}_{q_0}$ is an extremal tent space.
Indeed, multiplication is just a consequence of H\"older's inequality,
and factorization follows from the trick of taking powers: $|f|=|f|^{1-\theta}|f|^\theta$, 
with $0\leq\theta\leq 1$. 

Now the general factorization can be proved as follows.
With the H\"olderian relation $(H)$ in mind, factorizing $T^{p_0,r_0}_{q_0}$ 
through $F_3)$ in Theorem \ref{thm:factorizationend} into extremal tent spaces, 
using the known factorization for extremal tent spaces, 
and multiplying through $M_2)$ in Lemma \ref{lem:multiplicationend},
we then have
\begin{align*}
 T^{p_0,r_0}_{q_0}
&\rightarrow T^{p_0,\infty}_{\infty}\cdot 
T^{\infty,\infty}_{q_0} \cdot T^{\infty,r_0}_{\infty}\\
&\rightarrow T^{p_1,\infty}_{\infty}\cdot 
T^{p_2,\infty}_{\infty} \cdot T^{\infty,\infty}_{q_1} 
\cdot T^{\infty,\infty}_{q_2} \cdot T^{\infty,r_1}_{\infty} \cdot T^{\infty,r_2}_{\infty}
\rightarrow T^{p_1,r_1}_{q_1}\cdot T^{p_2,r_2}_{q_2}.
\end{align*}

Finally, the general multiplication can be proved as follows.
With the H\"olderian relation $(H)$ in mind, factorizing $T^{p_i,r_i}_{q_i}(i=1,2)$ 
through $F_3)$ in Theorem \ref{thm:factorizationend} into extremal tent spaces, 
using the known multiplication for extremal tent spaces, 
and multiplying through $M_2)$ in Lemma \ref{lem:multiplicationend},
we then have
\begin{align*}
 T^{p_1,r_1}_{q_1}\cdot T^{p_2,r_2}_{q_2}
&\rightarrow T^{p_1,\infty}_{\infty} 
\cdot T^{\infty,\infty}_{q_1} \cdot T^{\infty,r_1}_{\infty}
\cdot T^{p_2,\infty}_{\infty}\cdot T^{\infty,\infty}_{q_2} \cdot T^{\infty,r_2}_{\infty}\\
&\rightarrow T^{p_0,\infty}_{\infty}\cdot 
T^{\infty,\infty}_{q_0} \cdot T^{\infty,r_0}_{\infty}
\rightarrow T^{p_0,r_0}_{q_0}.
\end{align*}

The quasi-norm inequalities in each proof can be obtained by inspection.
\end{proof}

\section{quasi-Banach complex interpolation} \label{sec:inter}
We begin with a second look at the symbol ``$\leftrightarrow$'' 
for multiplication and factorization,
which we formulated in last section in Definition \ref{defn:mf}.

\begin{defn}
Given two quasi-Banach function lattices $X_1$ and $X_2$, 
we define their \textit{Calder\'{o}n's product} $X_1\bullet X_2$ 
as the class of $u\in L^0$ for which 
$$\|u\|_{X_1\bullet X_2}:=\inf\{\|v\|_{X_1}\|w\|_{X_2}\mid\,
|u|=|v||w|,\,v\in X_1,\,w\in X_2\}<\infty.$$
\end{defn}

Clearly, the usual product
$X_1\cdot X_2=\{vw\mid v\in X_1,\,w\in X_2\}$ 
is contained in the Calder\'{o}n's product $X_1\bullet X_2$.
In other words, $X_1\bullet X_2$ is the completion of $X_1\cdot X_2$
under the quasi-norm $\|\cdot\|_{X_1\bullet X_2}$.
Moreover, $X_0\leftrightarrow X_1\cdot X_2$ amounts to say $X_0=X_1\bullet X_2$,
where we interpret the equality by the equivalence of quasi-norms.

This new product $X_1\bullet X_2$, was first used 
by Calder\'{o}n in \cite{C} as an intermediate space 
for the complex interpolation of a couple of Banach function lattices $(X_1,X_2)$. 
For the underlying measure space $\Sigma=(\Omega,\mu)$, assume that
$\Omega$ is a complete separable metric space, 
and $\mu$ is a $\sigma$-finite Borel measure on $\Omega$. 
In a (most) natural extension of Calder\'{o}n's interpolation method to the quasi-Banach setting,
Kalton and Mitrea establish in \cite[Section 3]{KM} (see also \cite{K1}) that, 
for a couple of analytically convex separable 
quasi-Banach function lattices $(X_1,X_2)$ on $\Sigma$,  
there holds the generalized 
\textit{Calder\'{o}n's product formula} (see \cite[Theorem 3.4]{KM}) that
$$(X_1,X_2)_{\theta}=[X_1]^{1-\theta}\bullet [X_2]^\theta,\,0<\theta<1.$$ 
Here, $X$ analytically convex (A-convex for short) means that,
for any analytic\footnotemark\footnotetext{See \cite[p. 3911]{KM} 
for the precise definitions of analyticity and A-convexity.} 
function $\Phi:S=\{z\in \bC\mid\,\text{Re}\, z\in(0,1)\}\rightarrow X$, 
which is also continuous to the closed strip $\overline{S}=S\cup\pd S$, 
we have the \textit{maximum modulus principle}
$$\max_{z\in S}\|\Phi(z)\|_X \lesssim \max_{z\in \pd S} \|\Phi(z)\|_X.$$
Under this A-convexity requirement, $X_1+X_2$ is also A-convex,
and then Calder\'on's method adapts to the quasi-Banach case.
In the same spirit,
this analytical approach to the interpolation of quasi-Banach function lattices
was also considered in \cite{BeCe}, 
where the ambient A-convex space is not necessarily the usual $X_1+X_2$.
 
It was obtained in \cite{K} that
$X$ analytically convex is equivalent to $X$ $r$-convex for some $r>0$.
Here, $X$ \textit{(lattice) $r$-convex} means that, for any $n\in\bN_+$
and any $f_i\in X$, $i=1,\dots,n$, we have the inequality
$$\bigg\|\bigg(\sum_{i=1}^n|f_i|^r\bigg)^{1/r}
\bigg\|_X\leq \bigg(\sum_{i=1}^n\|f_i\|_X^r\bigg)^{1/r}.$$
This convexification/normalization process is trivial for Banach function lattice $X$,
as we can always take $r=1$ in the above inequality.
Thus for our purpose here, we can change A-convex to $r$-convex.

Now we turn to the separability issue. 
Recall that a Banach function lattice $X$ 
is said to satisfy the \textit{Fatou property} \cite[Remark 2 on p. 30]{LT}, 
or \textit{maximality in $L^0$}, if 
$$\forall\,0\leq f_n\in X\,\text{and}\,\sup_{n\in\bN_+}\|f_n\|_X<\infty,\, 
\text{with}\,f_n\uparrow f\in L^0\,\,\,\mu-\text{a.e.}$$
$$\Longrightarrow f\in X\,\text{and}\,\|f\|_X=\lim_{n\rightarrow\infty}\|f_n\|_X.$$
It was observed\footnotemark\footnotetext{In this regard, 
see also the second remark following Theorem 7.9 of \cite{KMM}, 
where $X_1$ and $X_2$ are assumed to be sequence spaces. In fact,
only the Fatou property is needed in the arguments there.} in \cite{K1} that, 
if both $X_1$ and $X_2$ satisfy the Fatou property,
we only need to assume for the interpolation that either $X_1$ or $X_2$ is separable.

For further information on the applicability of Calder\'{o}n's product formula,
see \cite[Section 3]{KM} and \cite[Section 7]{KMM} directly. 
\textit{Therefore, for two quasi-Banach function lattices $X_1$ and $X_2$,
if $X_i\,(i=1,2)$ is $r_i$-convex and has the Fatou property, 
and if either $X_1$ or $X_2$ is separable, 
then we have the desired interpolation realization:}
$$(X_1,X_2)_{\theta}=[X_1]^{1-\theta}\bullet [X_2]^\theta, \,0<\theta<1.$$

Let us apply these to tent spaces. 

\begin{lem} \label{lem:Fatou}
All the tent spaces $T^{p,r}_{q,\beta}$ have the Fatou property.
\end{lem}

\begin{proof}
This is an easy consequence of the monotone convergence theorem
and simple measure theoretic arguments.
\end{proof}

For $0<p_1,p_2\leq\infty$ and $\theta\in(0,1)$, 
define the \textit{$\theta$-H\"olderian triplet} 
$(p_1,p_2,(p_1,p_2)_\theta)$ by the relation 
$(p_1,p_2)_\theta^{-1}=(1-\theta)/p_1+\theta/p_2$, 
where we again admit $1/\infty=0$.

\begin{thm} \label{thm:Interpolation}
Let $0<\theta<1$. Suppose for $i\in\{0, 1, 2\}$, 
$T^{p_i,r_i}_{q_i,\beta_i}$ lies in the scale of weighted tent spaces 
with Whitney averages in Definition \ref{defn:tentspaces}.
Assume $\min(p_1,p_2)<\infty$ and the $\theta$-H\"olderian relation $(H)_\theta:$
$$p_0=(p_1,p_2)_\theta, \,\,q_0=(q_1,q_2)_\theta, \,\,r_0=(r_1,r_2)_\theta
\,\,\text{and}\,\, \beta_0=(1-\theta)\beta_1+\theta\beta_2.$$
Then under the Kalton-Mitrea complex interpolation method, we have
$$(T^{p_1,r_1}_{q_1,\beta_1},T^{p_2,r_2}_{q_2,\beta_2})_{\theta}
=T^{p_0,r_0}_{q_0,\beta_0}.$$
\end{thm}

\begin{proof}
With $(H)_\theta$ and Theorem \ref{thm:factorization}, we have  
$$T^{p_0,r_0}_{q_0,\beta_0}\leftrightarrow
T^{p_1/(1-\theta),r_1/(1-\theta)}_{q_1/(1-\theta),\beta_1(1-\theta)}\cdot
T^{p_2/\theta,r_2/\theta}_{q_2/\theta,\beta_2\theta},$$
which is equivalent to say
$$T^{p_0,r_0}_{q_0,\beta_0}=
T^{p_1/(1-\theta),r_1/(1-\theta)}_{q_1/(1-\theta),\beta_1(1-\theta)}\bullet
T^{p_2/\theta,r_2/\theta}_{q_2/\theta,\beta_2\theta}.$$

Under the condition $\min(p_1,p_2)<\infty$, 
at least one quasi-Banach function lattice in the interpolation couple 
$(T^{p_1,r_1}_{q_1,\beta_1},T^{p_2,r_2}_{q_2,\beta_2})$ is separable. 
And it follows from Minkowski's inequality that, for $i=1,2$, 
the quasi-Banach function lattice $T^{p_i,r_i}_{q_i,\beta_i}$ is $\min(\tau_i,1)$-convex,
where $\tau_i=\min(p_i,q_i,r_i)$. 
In fact, it suffices to apply $$\|f\|_{T^{p_i,r_i}_{q_i,\beta_i}}^{\tau_i}
=\||f|^{\tau_i}\|_{T^{p_i/\tau_i,r_i/\tau_i}_{q_i/\tau_i,\beta_i\tau_i}},\,i=1,2,$$
to the criterion of $r$-convexity, and notice that 
$T^{p_i/\tau_i,r_i/\tau_i}_{q_i/\tau_i,\beta_i\tau_i}(i=1,2)$ are Banach function lattices.
Using the generalized Calder\'{o}n's product formula, we have
\begin{align*}
 (T^{p_1,r_1}_{q_1,\beta_1},T^{p_2,r_2}_{q_2,\beta_2})_{\theta}
&=\big[T^{p_1,r_1}_{q_1,\beta_1}\big]^{1-\theta}\bullet
\big[T^{p_2,r_2}_{q_2,\beta_2}\big]^\theta\\
&=T^{p_1/(1-\theta),r_1/(1-\theta)}_{q_1/(1-\theta),\beta_1(1-\theta)}\bullet
T^{p_2/\theta,r_2/\theta}_{q_2/\theta,\beta_2\theta}=T^{p_0,r_0}_{q_0,\beta_0}.
\end{align*}
This proves the wanted complex interpolation formula. 
\end{proof}

The above interpolation result is plausibly new, 
in view of the novel Whitney averaging factor. 
For the tent spaces without Whitney averages and with $\beta=0$, 
the quasi-Banach complex interpolation
$$(T^{p_1}_{q_1},T^{p_2}_{q_2})_{\theta}=T^{p_0}_{q_0},\,0<\theta<1,$$
where $1/p_0=(1-\theta)/p_1+\theta/p_2$ and $1/q_0=(1-\theta)/q_1+\theta/q_2$,
was considered in \cite[Bernal]{Be}, by another analytical method 
and for the almost full range $0<p_1,p_2,q_1,q_2<\infty$.
For earlier results on the Banach complex interpolation, 
see the references in \cite{Be}.
Using the Kalton-Mitrea complex interpolation method,
\cite[Cohn-Verbitsky]{CV} recover the result in \cite{Be} and obtain additionally 
$$(T^{\infty}_{q},T^{p}_{\infty})_{\theta}=T^{p/\theta}_{q/(1-\theta)},\,0<\theta<1,$$
where $0<p,q<\infty$.
For the weighted analogue of \cite{CV}, 
see \cite[Hofmann-Mayboroda-McIntosh]{HMMc},
where the weight $\beta$ can also be any real number.

Here, by bringing in the endpoint space $T^\infty_\infty$,
we have under Theorem \ref{thm:Interpolation} and the coincidence theorem that,
for the full range $0<p_1,p_2,q_1,q_2\leq\infty$, we have
$$(T^{p_1}_{q_1},T^{p_2}_{q_2})_{\theta}=T^{p_0}_{q_0},\,0<\theta<1,$$
when $\min(p_1,p_2)<\infty$, $1/p_0=(1-\theta)/p_1+\theta/p_2$ and 
$1/q_0=(1-\theta)/q_1+\theta/q_2$. 
With this mild requirement $\min(p_1,p_2)<\infty$\footnotemark
\footnotetext{For the case $\min(p_1,p_2)=\infty$, 
there exist some results in a different context. 
For $\alpha\in[0,1]$ and the space of Carleson measures of order $\alpha$
$$V^\alpha:=\bigg\{d\mu\bigg|\sup_{B\subset\bRn}
\frac{|\mu|(\wh{B})}{|B|^\alpha}<\infty\bigg\},$$
the complex interpolation $(V^0,V^1)_{\alpha}$ was identified in \cite[Theorem 3-(ii)]{AB}
to a space which is strictly smaller than $V^\alpha$. 
In this respect, see also \cite{AM} and \cite{AM1} for relevant results.},
we then cover all the complex interpolation results 
obtained in \cite{Be}, \cite{CV} and \cite{HMMc}.

\section{Multipliers and standard duality} \label{sec:muldual}
Now we turn to the \textit{multiplier} issue, 
which from the \textit{multiplication} point of view,
is more straightforward than the quasi-Banach complex interpolation.

Similarly to the last section, we restrict ourselves to the setting of 
(Banach) function lattices, and the underlying measure space $\Sigma=(\Omega,\mu)$ 
is assumed to be complete and $\sigma$-finite. 
Here, ``complete'' is with respect to the measure, meaning that 
$$\forall\,E\subset\Omega, \mu(E)=0\Longrightarrow \forall\,E'\subset E, \mu(E')=0.$$
Recall that $L^0$ is the collection of all 
complex-valued $\mu$-measurable functions on $\Omega$.

\begin{defn}
Given two Banach function lattices $X_0$ and $X_1$, 
we say that $w\in L^0$ is a \textit{multiplier} from $X_1$ to $X_0$, 
if the associated multiplication mapping 
$$M_w:X_1\rightarrow X_0, v\mapsto vw$$ satisfies
$$\|M_w\|_{X_1\rightarrow X_0}:=\sup_{v\neq 0}\frac{\|vw\|_{X_0}}{\|v\|_{X_1}}<\infty.$$
Denote all the multipliers from $X_1$ to $X_0$ by $M(X_1,X_0)$, 
equipped with $$\|w\|_{M(X_1,X_0)}=\|M_w\|_{X_1\rightarrow X_0}.$$
\end{defn}

Before proceeding to our main results in this section,
we review a cancellation result concerning Calder\'{o}n's product.
It was obtained in \cite[Theorem 2.5 and Corollary 2.6]{S} that
for three Banach function lattices $\{E,F,G\}$ on $\Sigma$,
all with the Fatou property, we have the following \textit{cancellation formula} 
$$E\bullet F=E\bullet G\Longrightarrow F=G.$$
There also holds (see \cite[Theorem 2.8]{S}) that
$$F=M(E,E\bullet F),$$
if both $E$ and $F$ have the Fatou property. 
In particular situations, the above multiplier representation can also be found in 
\cite[Theorem 3.5]{CNS}, which served to prove the uniqueness theorem
of Calder\'{o}n-Lozanovskii's interpolation method.
We mention that in the literature, 
the construction of Calder\'{o}n for intermediate spaces was further 
investigated by Lozanovskii in a series of papers (\cite{L1}, \cite{L2}).

Let us apply these to our tent spaces. 

\begin{thm} \label{thm:multiplier}
With the same assumptions as in Theorem \ref{thm:factorization}
and $1\leq p_i,q_i,r_i\leq\infty$ for $i\in\{0,1,2\}$, 
we have the multiplier identification
$$T^{p_2,r_2}_{q_2,\beta_2}=
M(T^{p_1,r_1}_{q_1,\beta_1},T^{p_0,r_0}_{q_0,\beta_0}).$$ 
\end{thm}

\begin{proof}
For $i\in\{0,1,2\}$, $1\leq p_i,q_i,r_i\leq\infty$ 
implies $\tau_i=\min(p_i,q_i,r_i)\geq1$, 
thus $T^{p_i,r_i}_{q_i,\beta_i}$ is a Banach function lattice.
Using the multiplier representation cited above, 
with the Fatou property guaranteed by Lemma \ref{lem:Fatou}, we have
$$T^{p_2,r_2}_{q_2,\beta_2}=M(T^{p_1,r_1}_{q_1,\beta_1},
T^{p_1,r_1}_{q_1,\beta_1}\bullet T^{p_2,r_2}_{q_2,\beta_2})=
M(T^{p_1,r_1}_{q_1,\beta_1},T^{p_0,r_0}_{q_0,\beta_0}),$$
where the last equality is from Theorem \ref{thm:factorization}:
$T^{p_0,r_0}_{q_0,\beta_0}=T^{p_1,r_1}_{q_1,\beta_1}\bullet T^{p_2,r_2}_{q_2,\beta_2}$. 
\end{proof}

Finally, we look at the duality theory.
Given $\beta_0\in\bR$, we will consider the following $\beta_0$-weighted pairing
$$(f,h)_{\beta_0}:=\iint_{\bRt}f(y,t)h(y,t)t^{-\beta_0-1}dydt.$$
Let $p'$, $q'$ and $r'$ be the dual indice of $1\leq p,q,r\leq\infty$.

\begin{defn}
The $\beta_0$-weighted \textit{K\"othe dual} 
of the Banach $T^{p,r}_{q,\beta}$ is defined as
$$(T^{p,r}_{q,\beta})^*_{\beta_0}:=M(T^{p,r}_{q,\beta},L^1(\bRt,t^{-\beta_0-1}dydt))
=M(T^{p,r}_{q,\beta},T^{1,1}_{1,\beta_0}).$$
\end{defn}

Here, unlike the continuous functional dual $(\cdot)'$,
``K\"othe'' means the dual within the class of Banach function lattices.
For a general account on this aspect, see \cite{LT}.
By the \textit{standard duality}, 
we mean the (K\"othe) dual of the Banach $T^{p,r}_{q,\beta}$ 
when $1\leq p<\infty$, $\beta\in\bR$ and particularly $1\leq\min(q,r)\leq\max(q,r)<\infty$.

\begin{thm} \label{thm:duality}
Under the pairing $(\cdot,\cdot)_{\beta_0}$, we have the following standard duality 
$$T^{p',\,r'}_{q',\,\beta_0-\beta}=(T^{p,r}_{q,\beta})',\,1\leq p,q,r<\infty,\,\beta\in\bR.$$
\end{thm}

\begin{proof}
By Theorem \ref{thm:multiplier} and the definition of $(\cdot)^*_{\beta_0}$, we have
$$T^{p',\,r'}_{q',\,\beta_0-\beta}=M(T^{p,r}_{q,\beta},T^{1,1}_{1,\beta_0})=
(T^{p,r}_{q,\beta})^*_{\beta_0}\subset (T^{p,r}_{q,\beta})',$$
where the last inclusion follows from the straightforward
identification of multipliers to continuous linear functionals, 
through the pairing $(\cdot,\cdot)_{\beta_0}$.

For the converse, 
suppose that we are given a continuous linear functional $l$ on $T^{p,r}_{q,\beta}$.
Then whenever $K$ is a compact set in $\bRt$, and whenever $f$ is supported in $K$,
with $f\in L^r(K)$, then $\cW_r(f)\in T^p_{q,\beta}$ with 
$$\|f\|_{T^{p,r}_{q,\beta}}=\|\cW_r(f)\|_{T^p_{q,\beta}}\leq C_K\|f\|_{L^r}.$$
Here, $C_K$ is a constant which depends on the compact set $K$,
and also implicitly on the indice $p$, $q$, $r$ and $\beta$.
Thus $l$ induces a continuous linear functional on $L^r(K)$ 
and is representable by $h^K\in L^{r'}(K)$, as $1\leq r<\infty$.
Taking an increasing family of such $K$ which exhausts $\bRt$,
gives us an $h\in L^{r'}_{\loc}$ such that 
$$l(f)=(f,h)_{\beta_0}=\iint_{\bRt}f(y,t)h(y,t)t^{-\beta_0-1}dydt,$$
whenever $f\in L^r$ and has compact support. 
By density arguments, this representation of $l$ by $h$  
extends to all $f \in T^{p,r}_{q,\beta}$,
as we further have $1\leq p,q<\infty$.
By the representation through $(\cdot,h)_{\beta_0}$, we have
$\|l\|=\|M_h\|_{T^{p,r}_{q,\beta}\rightarrow T^{1,1}_{1,\beta_0}}$,
which means $$(T^{p,r}_{q,\beta})'\subset M(T^{p,r}_{q,\beta},T^{1,1}_{1,\beta_0})
=(T^{p,r}_{q,\beta})^*_{\beta_0}=T^{p',r'}_{q',\,\beta_0-\beta}.$$
This then proves the desired standard duality. 
\end{proof}

To end this section, we deduce as corollaries some corresponding known results
on multiplication, factorization and duality,
mainly obtained in the articles \cite[Coifman-Meyer-Stein]{CMS}, 
\cite[Cohn-Verbitsky]{CV} and \cite[Hyt\"onen-Ros\'en]{HR}.

\underline{Relation with Coifman-Meyer-Stein}.
For the standard duality, it was shown in \cite[Theorem 1-(b) and Theorem 2]{CMS} that
$$T^{p'}_2=(T^p_2)^*_{0}=(T^p_2)',\,1\leq p<\infty,$$
which upon using Theorem \ref{thm:coincidence} on the coincidence for $r=q=2$,
then corresponds to our Theorem \ref{thm:duality} in the particular case
$$T^{p',2}_{2,0}=\big(T^{p,2}_{2,0}\big)^*_{0}
=\big(T^{p,2}_{2,0}\big)',\,1\leq p<\infty.$$ 

By the \textit{Carleson duality}, 
we mean the continuous functional dual of $T^{p,r}_{q,\beta}$
for $1\leq p<\infty$, $\beta\in\bR$ and particularly 
$1\leq \min(q,r)\leq\max(q,r)=\infty$. 
Let $\underline{\wh{B}}:=\overline{\wh{B}}$ be the closed tent on base $B$, 
and denote the Carleson measures on $\overline{\bRt}$ by 
$$\overline{\cC}:=\bigg\{d\mu\bigg|\sup_{B\subset \bRn}
\frac{|\mu|\big(\underline{\wh{B}}\big)}{|B|}<\infty\bigg\}.$$ 
Let $\cN= T^1_\infty\cap C_{n.t.}$. 
The classical Carleson duality (\cite[Proposition 1]{CMS}) states that
$$\overline{\cC}=\big(\cN)'.$$
Obviously, our Theorem \ref{thm:duality} on standard duality 
can \textit{not} cover the Carleson duality. 
Nevertheless, we shall mention in Remark 6.2 a consequence of our  
method of proof toward factorization 
of bounded Borel measures on $\overline{\bRt}$ by Carleson measures.

\underline{Relation with Hyt\"onen-Ros\'en}.
To relate their notations, $N_{p,q}$ and $C_{p',q'}$ in \cite{HR} for Banach cases
are just the scales $T^{p,q}_{\infty,0}$ and $T^{p',q'}_{1,-1}$ here, 
and their duality claim is 
$$N_{p,q}=(C_{p',q'})',\,1<p<\infty,\, 1< q\leq\infty.$$
This \textit{Carleson (pre-)duality}, stated in \cite[Theorem 3.2]{HR}, 
then corresponds to our Theorem \ref{thm:duality} in the particular case
$$T^{p,r}_{\infty,0}=(T^{p',r'}_{1,-1})_{-1}^*=(T^{p',r'}_{1,-1})',
\,1<p<\infty,\,1< r\leq\infty.$$ 
At the multiplication side, Theorem 3.1 of \cite{HR} states
$$T^{r,r}_{r,-1/r}\leftarrow T^{p,q}_{\infty,0}\cdot T^{\wt{p},\wt{q}}_{r,-1/r},\,
1\leq r<\infty,\,r\leq p<\infty,\,r\leq q\leq\infty,$$ with $r=(p,\wt{p})_H=(q,\wt{q})_H$.
Again, this is a particular case of our Theorem \ref{thm:factorization}.

\underline{Relation with Cohn-Verbitsky}.
Under the coincidence theorem and Remark \ref{rem:nt},
part $F_2)$ in Theorem \ref{thm:factorizationend} 
for $r_0=q_0$ corresponds to Cohn-Verbitsky
$$T^{p_0}_{q_0}=T^{p_0,q_0}_{q_0}\rightarrow 
T^{p_0,\infty}_{\infty}\cdot T^{\infty,q_0}_{q_0}
=T^{p_0}_{\infty}\cdot T^{\infty}_{q_0}.$$
Meanwhile, with the help of $F_1)$ to produce Whitney multipliers, 
our result $F_3)$ is a further (polarized) factorization of the tent space $T^{p_0,r_0}_{q_0}$.
Of course, we also bring in the endpoint spaces 
$T^\infty_{\infty}$ and $T^{\infty,r_0}_{\infty}$,
which makes the statement broader.
Moreover, we continue with a multiplier discussion basing on the factorization result,
which is seemingly new even in the situation of classical tent spaces.

We also remark that the multiplication side of Theorem \ref{thm:factorization} 
covers Lemma 5.5 in \cite{AAInv} and Lemma 2.4.3 in \cite{M}.
To relate the notations again, the two tent spaces
$\cX$ and $\EP$ in \cite{AAInv}, originally introduced by Kenig-Pipher in \cite{KP}
and by Dahlberg in \cite{D} respectively, correspond to 
$T^{2,2}_{\infty,0}$ and $T^{\infty,\infty}_{2,0}$ here.
Our full scale $T^{p,r}_{q,\beta}$, mainly interested by $\cX^p:=T^{p,2}_{\infty,0}$ 
and $\cY_{\pm}^p:=T^{p,2}_{2,\frac{-1\pm1}{2}}$ for $p$
in some interval containing $2$, 
will be used as \textit{natural function spaces}
in part of a continuation work of \cite{AAInv},
where more backgrounds on boundary value problems of elliptic PDEs can be referred.

\section{Proof of Theorem \ref{thm:factorizationend} on factorization} \label{sec:proof}
To prove $F_3)$ it suffices to show $F_1)$ and $F_2)$ respectively.
Indeed, factorizing $T^{p_0,r_0}_{q_0}$ through $F_1)$ first, 
then using $F_2)$ yields $F_3)$ immediately.
Thus to prove Theorem \ref{thm:factorizationend}, we show $F_1)$ and $F_2)$ in order.

\begin{proof}[Proof of $F_1)$]
Let $W^*(y,t)$ and $\cW^*_{r}(\cdot)(y,t)$ be the Whitney box and 
the $L^r$-Whitney average associated to the point $(y,t)\in\bRt$, 
and to the Whitney parameters 
$$\alpha_1^*=\alpha_1\big(1+\alpha^{1/2}_2\big)^{-1}
\,\,\text{and}\,\, \alpha_2^*=\alpha_2^{1/2},$$ 
where $(\alpha_1,\alpha_2)$ is the pair of consistent Whitney parameters 
we used in Definition \ref{defn:tentspaces}.
Similarly, let $W^{**}$ and $\cW^{**}_{r}(\cdot)$ be the Whitney objects associated to 
$$\alpha_1^{**}=\alpha_1\Big[2\big(1+\alpha^{1/2}_2\big)\alpha^{1/4}_2\Big]^{-1}
\,\,\text{and}\,\, \alpha_2^{**}=\alpha_2^{1/4}.$$
Note that the two resulted pairs of Whitney parameters are also consistent, with
$$0<\alpha^{**}_1<\alpha_1^*<\alpha_1
<\alpha_2^{-1}<(\alpha_2^*)^{-1}<(\alpha_2^{**})^{-1}<1.$$

Moreover, for any $(y,t)\in\bRt$, we have the geometrical relations 
\begin{equation} \label{eqn:multiplier1}
\bigcup_{(z,s)\in W^*(y,t)} W^*(z,s)\subset W(y,t)
\end{equation}
and
\begin{equation} \label{eqn:multiplier2}
\bigcap_{(z,s)\in W^{**}(y,t)} W^*(z,s)\supset W^{**}(y,t).
\end{equation}
The verification on $\alpha_2^*$ and $\alpha_2^{**}$ is straightforward.
For the first inclusion, 
given any $(z,s)\in W^*(y,t)$ and any $(z_0,s_0)\in W^*(z,s)$, we have
$$|z_0-y|\leq|z_0-z|+|z-y|
<\alpha_1^*s+\alpha_1^*t<\alpha_1^*(\alpha_2^*+1)t=\alpha_1t,$$
which implies $(z_0,s_0)\in W(y,t)$.
For the second inclusion,
given any $(z_0,s_0)\in W^{**}(y,t)$ and any $(z,s)\in W^{**}(y,t)$, 
we have $$|z_0-z|\leq|z_0-y|+|y-z|
<2\alpha_1^{**}t<2\alpha_1^{**}\alpha_2^{**}s=\alpha_1^*s,$$
which implies $(z_0,s_0)\in W^*(z,s)$. 
This proves the two relations (\ref{eqn:multiplier1}) and (\ref{eqn:multiplier2}).

Now for any $u\in T^{p_0,r_0}_{q_0}$, we construct $v=\cW^*_{r_0}(u)$. 
Then we have from (\ref{eqn:multiplier1}) that
$$\sup_{(z,s)\in W^{*}(y,t)}\cW^*_{r_0}(u)(z,s)\lesssim \cW_{r_0}(u)(y,t)$$
is valid for any $(y,t)\in\bRt$, thus we know
$$\cW^*_{\infty}(v)\lesssim \cW_{r_0}(u)\,\,\,\text{and}\,\,\, 
\|\cW^*_{\infty}(v)\|_{T^{p_0}_{q_0}}\lesssim\|u\|_{T^{p_0,r_0}_{q_0}}.$$
For $w=u/\cW^*_{r_0}(u)$, we then have from (\ref{eqn:multiplier2}) that
$$\inf_{(z,s)\in W^{**}(y,t)}\cW^*_{r_0}(u)(z,s)\gtrsim \cW^{**}_{r_0}(u)(y,t)$$
is valid for any $(y,t)\in\bRt$, 
thus we know $$\cW^{**}_{r_0}(w)\lesssim1\,\,\,\text{and}\,\,\, 
\|\cW^{**}_{r_0}(w)\|_{T^\infty_\infty}\lesssim1.$$ 

Using the change of Whitney parameters equivalence in Observation \ref{obs:changewhit}, 
$u=vw$ is then the desired factorization for 
$T^{p_0,r_0}_{q_0}\rightarrow T^{p_0,\infty}_{q_0}\cdot T^{\infty,r_0}_\infty$,
$0<p_0,q_0,r_0\leq\infty$.
\end{proof}

\begin{proof}[Proof of $F_2)$]
Observe that we can suppose $0<\max(p_0,q_0)<\infty$.
In fact, nothing has to be done if $p_0=\infty$, 
and the case $q_0=\infty$ is already included in $F_1)$.

We base our arguments on the constructive proof in \cite{CV}.
From the consistency of Whitney parameters, we have $0<\alpha_1<\alpha_2^{-1}<1$.
Then the following relations  
\begin{equation} \label{eqn:whitneygeo1}
\bigcap_{(z,s)\in W(y,t)}B(z,s)\supset B(y,(\alpha_2^{-1}-\alpha_1)t)
\end{equation}
and
\begin{equation} \label{eqn:whitneygeo2}
\bigcup_{(z,s)\in W(y,t)}B(z,s)\subset B(y,(\alpha_2+\alpha_1)t)
\end{equation}
hold for any $(y,t)\in \bRt$. In fact, for the verification of the first inclusion,
given any $x\in B(y,(\alpha_2^{-1}-\alpha_1)t)$
and any $(z,s)\in W(y,t)$, we compute as follow
$$|x-z|\leq|x-y|+|y-z|<(\alpha_2^{-1}-\alpha_1)t+\alpha_1t<s,$$
which implies $x\in B(z,s)$.
Similarly, to verify the second inclusion,
given any $(z,s)\in W(y,t)$ and any $x\in B(z,s)$, we compute as follow
$$|x-y|\leq|x-z|+|z-y|<s+\alpha_1t<(\alpha_2+\alpha_1)t,$$
which implies $x\in B(y,(\alpha_2+\alpha_1)t)$.
This proves the two relations (\ref{eqn:whitneygeo1}) and (\ref{eqn:whitneygeo2}).

As $0<\max(p_0,q_0)<\infty$, the tent space $T^{p_0,r_0}_{q_0}$ lies in Category $A)$
and can be determined by the conical functional $\cA_{q_0}$.
Therefore, $\tilde{u}=\cA_{q_0}(\cW_{r_0}(u))\in L^{p_0}(\bRn)$.
Denote by $P_0[h](y,t)$ the average of $h$ on $B(y,t)\subset\bRn$,
and construct $v=P_0[\tilde{u}^{\tilde{p}}]^{1/\tilde{p}}$ for some $\tilde{p}<p_0$.
Let $\alpha^*=\alpha_2+\alpha_1>1$, 
then by (\ref{eqn:whitneygeo2}), for any $(y,t)\in\bRt$ 
$$\sup_{(z,s)\in W(y,t)}v(z,s)\lesssim v(y,\alpha^*t)=:v^*(y,t).$$ 
Thus we have $\cW_{\infty}(v)(y,t)\lesssim v^*(y,t)$, and there holds
\begin{align*}
\cN(\cW_{\infty}(v))(x)\lesssim\cN(v^*)(x)\leq
\cM(\tilde{u}^{\tilde{p}})^{1/\tilde{p}}(x),\,\forall\,x \in\bRn,
\end{align*}
where $\cN$ is the non-tangential maximal functional, 
$\cM$ is the Hardy-Littlewood maximal operator and 
the last estimate follows from the fact
$$\bigcap_{(y,t)\in\Gamma(x)}B(y,\alpha^*t)\ni x,\,\forall\,x \in\bRn.$$
As $p_0/\tilde{p}>1$, then by maximal theorem, we have
$$\|v\|_{T^{p_0,\infty}_{\infty}}\lesssim
\|\cM(\tilde{u}^{\tilde{p}})^{1/\tilde{p}}\|_{L^{p_0}}
\lesssim\|\tilde{u}\|_{L^{p_0}}=\|u\|_{T^{p_0,r_0}_{q_0}}.$$

Now we turn to $w=u/v$. Let $\alpha_*=\alpha_2^{-1}-\alpha_1\in(0,1)$, 
then by (\ref{eqn:whitneygeo1})
$$\inf_{(z,s)\in W(y,t)}v(z,s)\gtrsim v(y,\alpha_*t)$$ 
is valid for any $(y,t)\in\bRt$. By H\"{older}'s inequality, there holds
\begin{equation} \label{eqn:holder}
\|h^{-1}\|^{-1}_{L^q(d\nu)}\leq\|h\|_{L^r(d\nu)},
\,\forall\, q>0,\,\forall\,r>0, 
\end{equation}
when $d\nu$ is a probability measure on $\bRn$. 
Applying this estimate with $h=\tilde{u}$, $r=\wt{p}$,
$q=q_0$ and $d\nu(x)=|B(y,\alpha_*t)|^{-1}\chi_{B(y,\alpha_*t)}(x)dx$, 
we have for any $(y,t)\in\bRt$
\begin{align*}\inf_{(z,s)\in W(y,t)}v(z,s)
 &\gtrsim P_0[\tilde{u}^{\tilde{p}}]^{1/\tilde{p}}(y,\alpha_*t)\\
&\geq P_0[\tilde{u}^{-q_0}]^{-1/q_0}(y,\alpha_*t)
\gtrsim P_0[\tilde{u}^{-q_0}]^{-1/q_0}(y,t),
\end{align*}
where the last estimate follows from $0<\alpha_*<1$ and $-1/q_0<0$.
We write $\|\cdot\|_c=\|\cdot\|_{T^\infty_{1,-1}}$ 
for the Carleson norm of measurable functions on $\bRt$,
and let $$d\mu(y,t)=\mu(y,t)dydt=\cW_{r_0}(u)^{q_0}(y,t)t^{-1}dydt.$$
The above pointwise estimates on $v$ further imply
\begin{align*}
\|\cW_{r_0}(u/v)\|_{T^{\infty}_{q_0}}
&\lesssim \|P_0[\tilde{u}^{-q_0}]^{1/q_0}\cW_{r_0}(u)\|_{T^{\infty}_{q_0}}\\
&= \|P_0[\tilde{u}^{-q_0}]\mu\|^{1/q_0}_{T^\infty_{1,-1}}
=\|P_0[\cA_{1}(\mu(y,t)t)^{-1}]\mu\|^{1/q_0}_c\lesssim 1.
\end{align*}
In the last estimate, we used the lemma below. 

Therefore, we can conclude the proof of $F_2)$.
\end{proof}

We record down the missing part in estimating 
$\|P_0[\cA_{1}(\mu(y,t)t)^{-1}]\mu\|_c\lesssim 1$.
For a non-negative measure $d\mu$ on $\overline{\bRt}$, 
denote its (free) balayage by
$$\overline{\cA}(d\mu)(x):=\iint_{\Gamma(x)}\frac{d\mu(z,s)}{s^{n}},\,x\in\bRn.$$
This way, we can reconstruct from the boundary value 
$\overline{\cA}(d\mu)$ its (free) extension
$$E(d\mu)(y,t):=P_0[\overline{\cA}(d\mu)^{-1}](y,t),\,\forall\,(y,t)\in\bRt.$$
Thus in the desired estimate, with $d\mu(y,t)=\mu(y,t)dydt$ supported in $\bRt$, we have 
$$P_0[\cA_{1}(\mu(z,s)s)^{-1}](y,t)\mu(y,t)dydt=E(d\mu)(y,t) d\mu(y,t).$$

The next lemma is very simple and can be found in \cite[Lemma 2.2]{CV}, 
or one can refer to \cite{AB} directly.
For the completeness, we still provide an argument here.
Recall that $\wh{\underline{B}}$ denotes the closed tent with base $B\subset\bRn$.

\begin{lem} \label{lem:bala}
For any non-negative measure $d\mu$ on $\overline{\bRt}$, we have
$$\|E(d\mu) d\mu\|_{\overline{\cC}}:=\sup_{B\subset\bRn}\frac{1}{|B|}
\iint_{\wh{\underline{B}}}E(d\mu) d\mu\lesssim1.$$
\end{lem}

\begin{proof}
For any ball $B\subset \bRn$, we can estimate by Fubini's theorem that
\begin{align*}
\iint_{\wh{\underline{B}}}\bigg[
\frac{1}{|B(y,t)|}&\int_{B(y,t)}\overline{\cA}(d\mu)^{-1}(x)dx\bigg]d\mu(y,t)\\
&\simeq\iint_{\wh{\underline{B}}}\bigg[\int_{B(y,t)}
\overline{\cA}(d\mu)^{-1}(x)dx\bigg]\frac{d\mu(y,t)}{t^n}\\
&=\int_{\bRn}\overline{\cA}(d\mu)^{-1}(x)
\bigg[\iint_{\wh{\underline{B}}\cap\Gamma(x)}\frac{d\mu(y,t)}{t^n}\bigg]dx\\
&\leq\int_B\overline{\cA}(d\mu)^{-1}(x)\overline{\cA}(d\mu)(x)dx= |B|.
\end{align*}
Taking a supremum over balls $B\subset\bRn$ then proves the Carleson estimate.
\end{proof}

\begin{rem} \label{rem:mesfac}
Denote by $\overline{\cV}$ 
the class of bounded (signed and complex) Borel measures on $\overline{\bRt}$.
Note that the above lemma also implies the factorization
$$\overline{\cV}\rightarrow (T^1_\infty\cap C_{n.t.})\cdot \overline{\cC},$$
while the multiplication side $\overline{\cV}\leftarrow (T^1_\infty\cap C_{n.t.})
\cdot \overline{\cC}$ is just the Carleson's inequality 
(see \cite[p. 63]{Stein} for example).
Indeed, for $d\mu$ bounded on $\overline{\bRt}$,
$$|d\mu|=E(|d\mu|)^{-1}\cdot E(|d\mu|)|d\mu|$$ is then the desire factorization.
First, using the lemma above, we have
$$\|E(|d\mu|)|d\mu|\|_{\overline{\cC}}\lesssim1.$$
And by (\ref{eqn:holder}), we see for any $(y,t)\in\bRt$ that 
$$E(|d\mu|)^{-1}(y,t)\leq 
\bigg(\frac{1}{|B(y,t)|}\int_{B(y,t)}\overline{\cA}(|d\mu|)^{p_0}(x)dx
\bigg)^{1/p_0},\,0<p_0<1.$$ 
Then for any $x\in\bRn$, we have 
$$\cN(E(|d\mu|)^{-1})(x)\leq \cM(\overline{\cA}(|d\mu|)^{p_0})^{1/p_0}(x),$$ 
and by Lebesgue's theorem $E(|d\mu|)^{-1}\in C_{n.t.}$.
By maximal theorem, we also have $E(|d\mu|)^{-1}\in T^1_\infty$, with the factorization estimate
$$\|E(|d\mu|)^{-1}\|_{T^1_\infty}\lesssim\|\overline{\cA}(|d\mu|)\|_{L^1}
\simeq|\mu|\big(\overline{\bRt}\big).$$
\end{rem}

\begin{rem} \label{rem:nt}
In $F_1)$, the case $r_0=\infty$ is trivial.
Suppose $0<r_0<\infty$ and $\cW_{r_0}(u)\in C_{n.t.}$.
As the constructed $v=\cW^*_{r_0}(u)$ is continuous 
and satisfies $\cW^*_{\infty}(v)\lesssim \cW_{r_0}(u)$,
we have $\cW^*_{\infty}(v)\in C_{n.t.}$ after using the fact (\ref{eqn:multiplier1})
$$\lim_{\Gamma(x)\ni(y,t)\rightarrow x}W^*(y,t)=
\lim_{\Gamma(x)\ni(y,t)\rightarrow x}W(y,t)=x,\,\forall\,x\in\bRn,$$
and the dominated convergence theorem.  

In $F_2)$, if $0<\max(p_0,q_0)<\infty$,
we can also verify that $\cW_\infty(v)$ is continuous in $\bRt$ 
and has the property of non-tangential convergence. In fact, 
$$v^{\wt{p}}(y,t)=|B(y,t)|^{-1}\int_{B(y,t)}\wt{u}^{\wt{p}}(x)dx,\,\forall\,(y,t)\in\bRt,$$ 
where $\wt{u}\in L^{p_0}$ and $p_0>\wt{p}$. 
Then $v\in C_{n.t.}$ follows from Lebesgue's theorem. As
$$v(y,\alpha_*t)\lesssim\inf_{(z,s)\in W(y,t)}v(z,s)\leq 
\sup_{(z,s)\in W(y,t)}v(z,s)\lesssim v(y,\alpha^*t)$$ hold true for any $(y,t)\in\bRt$, 
we then have 
$$\cW_\infty(v)=\sup_{(z,s)\in W(y,t)}v(z,s)\in C_{n.t.},$$
which is an easy consequence of the dominated convergence theorem. In all,
the constructed factorization $v$ is in 
$(T^{p_0,\infty}_\infty \cap C_{n.t.})=(T^{p_0}_\infty \cap C_{n.t.})$.
\end{rem}

\section*{Acknowledgement}
As part of the author's thesis project,
the current paper is written under the guidance of Prof. Pascal Auscher,
whose patience is greatly acknowledged.
The author would like to thank Prof. Auscher 
and also Dr. Henri Martikainen for helpful discussions.
This research is supported in part by the ANR project 
``Harmonic Analysis at its Boundaries'', ANR-12-BS01-0013-01.
The author would also like to thank Prof. Dachun Yang and 
Dr. Jonathan Sondow for their continuous encouragements.

\bibliographystyle{amsplain}

\end{document}